\documentclass[10pt,draftcls,onecolumn]{IEEEtran}

\usepackage{graphicx}      
\usepackage{mathtools, verbatim}
\usepackage{amsmath, tabularx}
\usepackage{amsfonts}
\usepackage{amssymb}
\usepackage{amsbsy}
\usepackage{algorithm,algpseudocode}
\usepackage{enumerate}
\usepackage{enumitem}
\usepackage{color}
\usepackage{accents}
\usepackage{caption}
\usepackage{subcaption}
\usepackage{hyperref}
\usepackage{makecell}

\usepackage{latexsym}
\usepackage{cite}
\usepackage[usenames,dvipsnames]{xcolor}

\newtheorem{remark}{Remark}

\newtheorem{theorem}{Theorem}
\newtheorem{definition}{Definition}

\newtheorem{proposition}{Proposition}
\newtheorem{lemma}{Lemma}
\newtheorem{proof}{Proof}

\algdef{SE}[EVENT]{Event}{EndEvent}[1]{\textbf{upon event}\ #1\ \algorithmicdo}{\algorithmicend\ \textbf{event}}%
\algtext*{EndEvent}

\newcommand{\eg}{{\it e.g.}}
\newcommand{\ie}{{\it i.e.}}
\newcommand{\reals}{{\mathbb{R} }}
\newcommand{\tr}{{T }}

\newcommand{\dom}{\operatorname{ dom}}
\newcommand{\epi}{\operatorname{ epi}}
\newcommand{\conv}{\operatorname{ conv}}
\newcommand{\argmin}{\operatorname{ argmin}}
\newcommand{\prox}{\operatorname{\bf prox}}
\newcommand{\refl}{\operatorname{\bf refl}}

\newcommand{\argmax}{\operatorname{ argmax}}
\newcommand{\LeftEqNo}{\let\veqno\@@leqno}

\begin{document}
%
\title{Communication reduction in distributed optimization via estimation of the proximal operator}
%
%
%

\author{Giorgos~Stathopoulos,~\IEEEmembership{Member,~IEEE,}
        and~Colin~N.~Jones,~\IEEEmembership{Member,~IEEE.}
\thanks{The authors are with the Laboratoire d'Automatique, \'Ecole Polytechnique F\'ed\'erale de Lausanne (EPFL), Lausanne, Switzerland. E-mail: \{georgios.stathopoulos,colin.jones\}@epfl.ch .}
\thanks{The research leading to these results has received funding from the European Research Council under the European Union's Seventh Framework Programme (FP/2007-2013) / ERC Grant Agreement n. 307608: BuildNet.}
\thanks{This project has received funding from the European Research Council (ERC) under the European Union’s Horizon 2020 research and innovation programme (grant agreement No 755445)}
}

\markboth{Journal of \LaTeX\ Class Files,~Vol.~14, No.~8, August~2015}%
{Stathopoulos \MakeLowercase{\textit{et al.}}: Title}

\maketitle

\begin{abstract}
We introduce a reduced-communication distributed optimization scheme based on estimating
the solution to a proximal minimization problem. Our proposed setup involves a group of
agents coordinated by a central entity, altogether operating in a collaborative framework.
The agents solve proximal minimization problems that are hidden from the central
coordinator. The scheme enables the coordinator to construct a convex set within which the agents'
optimizers reside, and to iteratively refine the set every time that an agent is queried.
We analyze the quality of the constructed sets by showing their connections to the
$\epsilon$-subdifferential of a convex function and characterize their size. We prove convergence
results related to the solution of such distributed optimization problems and we devise a
communication criterion that embeds the proposed scheme in the Alternating Direction Method of
Multipliers (ADMM). The developed scheme demonstrates significant communication reduction when applied
to a microgrid setting.
\end{abstract}

\begin{IEEEkeywords}
Distributed optimization, Network optimization, ADMM, proximal operator, Moreau envelope, $\epsilon$-subdifferential
\end{IEEEkeywords}
\IEEEpeerreviewmaketitle

\section{Introduction}{\label{sec::motivation}}
The presence of a population of independent agents with a global coordinator, along with the fact that certain data is supposed to remain private to the agents, call for distributed optimization schemes. Methods like the Proximal Gradient Method (PGM), the Alternating Direction Method of Multipliers (ADMM) and several others are suitable for these problems (see~\cite{admm_distr_stats} and the references therein).

In the majority of these schemes, the local subproblems that need to be solved are cast as \emph{proximal minimization problems}, a term used to describe optimization problems that are regularized by the addition of a properly scaled quadratic term in their objective~\cite{parikh2014proximal}. In the course of the execution, the agents need to communicate the solutions to the proximal minimization subproblems to the coordinator, who will, in turn, manipulate the agents' objectives by broadcasting incentives that skew their local policies toward the global target.

In such multi-agent frameworks, extensive communication might be undesirable for a variety of reasons. Such reasons involve the existence of delays due to a weak network, the fact that the agents might run on energy-limited resources which are drained rapidly with frequent activations, or speeding up computation by skipping insignificant agent updates. It would, therefore, be useful if the coordinator could `guess' the optimizers of its agents and base the selection of the agent to update on the satisfaction of some criterion.

This work is an extension of~\cite{learnProx_ECC}, where a reduced communication framework for distributed optimization problems has been proposed. This is achieved by estimating a convex set containing the solution of a proximal minimization problem. Construction of the set is based on the theory of \emph{the Moreau envelope function} and its important connections with the proximal operator. The structural properties of the Moreau envelope allow the coordinator to bound the set of all possible optimizers to a convex set, explicitly described as the intersection of ellipsoids and iteratively refined every time that a communication round occurs. Subsequently, the coordinator can make a guess regarding the solution of the agent's optimizer by choosing a value from the constructed set and spare a communication round provided that the guess is adequately good.

In this work we further analyze the properties of the constructed ellipsoidal sets that contain the unknown solution of the proximal problem. We show that \emph{the sets can be recovered as limiting cases of $\epsilon$-subdifferential sets associated to functions that upper bound the Moreau envelope function} and that \emph{they are optimal with respect to the tightest upper bound of the (unknown) Moreau envelope.} We model the effect incurred by the lack of communication as an error in the solution of the optimization problem. By analyzing this result in the context of fixed-point iterations with errors, \emph{we prove asymptotic convergence of the sequence to an optimizer.} In addition, \emph{we devise a communication test for ADMM that enables adoption of the proposed framework}, and that significantly extends the applicability of the approach in comparison to~\cite{learnProx_ECC}, where only the projected gradient method was considered.

The manuscript is organized as follows: The related work is cited in Section~\ref{sec::related}, while the notation is explained in Section~\ref{sec::notation}. In Section~\ref{sec::problem} we define the problem we are interested in solving and give an overview of the proposed solution approach. In Section~\ref{sec::theory} we provide a step-by-step construction of the set where the solution lies. We show how it is related to the $\epsilon$-subdifferential set of a quadratic upper bound to the envelope function and prove that it is optimal with respect to some metric. We also propose an approximation of the set that can be used for computations. In Section~\ref{sec::communication} we couple the proposed scheme to ADMM and devise a certification test that decides when the gradient set should be updated, namely when a communication round should be triggered. Subsequently, convergence results are proven in Section~\ref{sec::convergence}. Finally, in Section~\ref{sec::application} we provide evidence about the performance of the proposed scheme by solving a load sharing problem for microgrids.

\section{Related work }{\label{sec::related}}
A comprehensive theory of the Moreau envelope function can be found in~\cite{RockWets98}, while the algorithmic aspects of the proximal iteration and its gradient descent interpretation are analyzed in~\cite{MoreYos}. In our proposed scheme, the convex set within which the possible optimizers reside can be essentially constructed by making use of a fundamental gradient inequality as will be shown in Section~\ref{sec::theory}. Interestingly, a similar line of thinking has been followed in~\cite{Taylor2017} in order to compute the exact worst-case performance of fixed-step first-order methods for unconstrained optimization of smooth convex functions. We are not aware of other works that seek to `learn' the solutions to proximal minimization problems, or, in fact, to any other optimization problem. In the context of communication avoidance in distributed optimization, the work~\cite{ComRedProx} proposes a scheme for reducing communication rounds for a stochastic version of the iterative soft thresholding algorithm. In a similar but rather more generic manner, the authors in~\cite{Ma:2017} propose a communication-efficient distributed-optimization framework for large-scale machine learning applications. 

\section{Notation}{\label{sec::notation}}
Consider an (extended-real-valued) closed proper convex function
$f:\reals^n\mapsto\reals\cup\{+\infty\}$. Adopting the definitions from~\cite{bertsekas03},
we denote the \emph{effective domain} of $f$ by $\dom\;(f)$, where
\[
\dom\;(f) = \left\{ x\in\reals^n\mid f(x)<+\infty\right\},
\]
and its \emph{epigraph} by $\epi\;(f)$. The epigraph is the subset of $\reals^{n+1}$ given by
\[
\epi\;(f) = \left\{ (x,s)\mid x\in\reals^n, s\in\reals, f(x) \leq s\right\}.
\]
The subdifferential of $f$ at a point $x$ is denoted by $\partial f(x)$.\\
The \emph{proximal operator} $\prox{_{\gamma f}}:\reals^n\mapsto\reals^n$ evaluated at
$z$ is denoted as $\prox{_{\gamma f}}(z)=\underset{x}{\argmin}\left\{f(x)+(1/(2\gamma))\|x-z\|^2\right\}$, $\gamma>0$.
The \emph{indicator function} of a closed convex set $\mathcal{Z}$ is denoted as
\[
\delta(z\mid\mathcal{Z}) = \left\{ \begin{array}{ll} 0 & z\in\mathcal{Z} \\
              \infty & \mbox{otherwise}.
              \end{array} \right.
\]
When $f=\delta(\cdot\mid\mathcal{Z})$, $\prox{_{\gamma f}}(z)=\mathcal{P}_{\mathcal{Z}}(z)$,
where $\mathcal{P}_{\mathcal{Z}}$ denotes the projection of $z$ onto the set $\mathcal{Z}$.\\

Another useful operator is the \emph{reflection operator}, denoted as $\refl{_{\gamma f}}:\reals^n\mapsto\reals^n$ and
defined as $\refl{_{\gamma f}}=2\prox{_{\gamma f}}-I$, where $I$ is the identity matrix.\\
Finally, we denote the \emph{convex conjugate} of $f$ by the function
\[
f^\star(\lambda) = \underset{x\in\reals^n}{\sup}\left\{\langle x,\lambda\rangle - f(x)\right\}, \; \lambda\in\reals^n.
\]

\section{Problem description}{\label{sec::problem}}

\subsection{The problem}
Consider an optimization problem involving $N$ agents and a coordinator that takes the form
\begin{equation}{\label{eq::distopt}}
\begin{aligned}
&{\text{minimize}} && h(x)+ \sum_{i=1}^Nf_i(x_i)\enspace,
\end{aligned}
\end{equation}
where $f_i:\reals^{n_i}\mapsto\reals\cup\{+\infty\},\;i=1,\ldots,N$ are arbitrary closed proper convex functions,
private to the corresponding agents, and $h:\reals^n\mapsto\reals\cup\{+\infty\}$, $n=\sum_{i=1}^Nn_i$, is the
convex objective of the coordinator.

In distributed optimization settings like the one above, the population of agents and the coordinator cooperate so as to solve problem~\eqref{eq::distopt}. In order to achieve this, it is often the case that a \emph{decomposition method} is applied, \eg, the Proximal Gradient Method, ADMM etc. Since certain agents' data is supposed to remain private in several settings, each agent iteratively solves a proximal minimization problem of the form 
\begin{equation}{\label{eq::prox_def_many}}
 \prox{_{\gamma f_i}}(z_i^k)=\underset{x_i\in\reals^{n_i}}{\argmin}\left\{f_i(x_i)+\frac{1}{2\gamma}\|x_i-z_i^k\|^2\right\}\enspace.
\end{equation}
During the execution of the employed decomposition algorithm, a sequence of points $\{z_i^k\}_{k\in\mathbb{N}}$ at which the proximal minimization problem is solved is generated by the coordinator. These points correspond to the coordinator's intervention in order to steer the agents toward the minimizer of the (global) objective $h$. The response vector $\prox{_{\gamma f_i}}(z_i^k)$ is then transmitted from the agent to the coordinator. Such iterative schemes can be seen as \emph{querying mechanisms}, where the coordinator broadcasts some value $\{z_i^k\}$ and the agent responds with another.

In setups where communication is expensive, the coordinator would like to avoid excessive querying of the agents, especially if the elicited information is not of importance. \emph{Our goal is to estimate, to the best possible accuracy, the optimizer
of~\eqref{eq::prox_def_many} for a given (arbitrary) sequence $\{z_i^k\}_{k\in\mathbb{N}}$, without having to solve
the optimization, namely we want to `guess' the solution of a (local) proximal minimization problem from the coordinator's perspective, \ie, without knowledge of the funtction $f_i$.}

\subsection{The approach}
Focusing on just one agent for the purpose of the analysis, we rewrite the proximal minimization problem as 
\begin{equation}{\label{eq::prox_def}}
 \prox{_{\gamma f}}(z)=\underset{x\in\reals^{n}}{\argmin}\left\{f(x)+\frac{1}{2\gamma}\|x-z\|^2\right\}\enspace.
\end{equation}
We will try to introduce some structure in~\eqref{eq::prox_def} so as to estimate it. 
The value function of \eqref{eq::prox_def} is also known as the Moreau envelope of $f$, defined as
 \begin{equation}{\label{eq::moreau}}
   f^\gamma(z) = \min_{x\in\reals^n} \left\{ f(x) + \frac{1}{2\gamma}\|x-z\|^2\right\}\enspace.
\end{equation}
When $f$ is a closed proper convex function, the Moreau envelope is \emph{convex and differentiable, with
Lipschitz continuous gradient with constant $1/\gamma$}.
Moreover, \emph{the set of minima of~$f$ and of $f^\gamma$ coincide}.
A discussion about the envelope and its properties can be found in~\cite[Section~5.1]{bertsekas_cvx_algos}.
It is also shown in~\cite[Proposition~5.1.7]{bertsekas_cvx_algos} that
\emph{the unique solution to the proximal minimization~$x^\gamma(z)=\prox_{\gamma f}(z)$ can be written as
   \begin{equation}{\label{eq::optimizer}}
     x^\gamma(z) = z - \gamma \nabla f^\gamma(z)\enspace,
  \end{equation}
  \ie,
it is the point at which the gradient iteration of $f^\gamma$, evaluated at $z$, lands}.

It is evident from~\eqref{eq::optimizer} that, given an arbitrary point $z$,
the gradient $\nabla f^\gamma(z)$
is all that is needed in order to reconstruct the optimizer of~\eqref{eq::prox_def}.
Therefore, the coordinator aims at constructing
\emph{a set of possible gradients at $z$ and a way to evaluate
the worst case gradient (how far away the estimate can be from the actual gradient) contained in the set}.

Our approach is to construct such a set and to continually refine it (shrink it)
every time that an exchange occurs between the coordinator and the agent, namely every
time that~\eqref{eq::prox_def} is solved for a point $z$. We are going to refer to points at which
the problem is actually solved, as \emph{query points}. The coordinator's goal is to estimate the solution of~\eqref{eq::prox_def} at a point of interest by using the solution at previous query points.


Assume the existence of a number of generated query points $z_j,\;j\in\mathcal{J}$.
In what follows, we proceed to construct a set that contains the gradient $\nabla f^\gamma(v)$ at the point of interest $v$
in six distinct steps described below.
\begin{enumerate}
\item We first devise a family of (local) nested sets that all contain $\nabla f^\gamma(v)$. Each family is centered around its corresponding query point $z_j$.
\item We show that the family of sets can be represented by a convex inequality.
\item We find the smallest set in the family for each query point $z_j$.
\item We devise a (global) set within which $\nabla f^\gamma(v)$ resides by taking the intersection of the smallest sets for all $z_i$'s.
\item We derive the relation between the global set and the tightest known function that upper bounds $f^\gamma(v)$.
\item We derive an approximation of the global set that can be used for numerical computations. 
\end{enumerate}

\section{Estimating the solution to a proximal minimization problem}{\label{sec::theory}}
\paragraph{Quadratic upper bound} Since we want to locate the unknown gradient $\nabla f^\gamma(v)$ in a set, what comes to mind is a notion of \emph{subdifferential}
of some function at $v$. The function we are about to use is a quadratic upper bound on the Moreau envelope,
implied by the Lipschitz gradient property of $f^\gamma$.
\begin{proposition}{\label{prop::upperQuadratic}}
The Moreau envelope $f^\gamma$ admits a quadratic upper bound in terms of the linear approximation of $f^\gamma$ based on the gradient
at any $v\in\reals^n$, \ie,
\begin{equation}{\label{eq::quadratic_bound}}
\overline{f}^\gamma(z;v) = f^\gamma(v) + \langle \nabla f^\gamma(v),z-v\rangle + \frac{1}{2\gamma}\|z-v\|_2^2, \quad z,v\in\reals^n\enspace,
\end{equation}
(see, \eg, \cite[Section~6.1]{bertsekas_cvx_algos}).
\end{proposition}

The quadratic upper bounding function~\eqref{eq::quadratic_bound} will enable us to
compute the desired set.
Consider a query point, namely $z_1$. We will carry out the remaining analysis for the case that only one query point has been generated,
while the point of interest will be $z=v$. The point $z_1$ is associated to the quadratic
approximation $\overline{f}^\gamma(z;z_1)$ that upper bounds $f^\gamma$.  Looking at
Figure~\ref{subfig::epis:2}, we observe that the point $(v,f^\gamma(v))$ lies below the epigraph of
$\overline{f}^\gamma(z;z_1)$, while the blue lines \emph{are the tangent hyperplanes of 
$\mathrm{epi}\;(\overline{f}^\gamma(z;z_1))$ at $v$}. This brings up the notion of the $\epsilon$-subdifferential
set (\cite[Section~3.3]{bertsekas_cvx_algos}), the definition of which is given below:
\begin{definition}[$\epsilon$-subgradient]{\label{def::eps_subdiff}}
 Given a proper convex function $f:\reals^n\mapsto\reals\cup\{+\infty\}$ and a scalar $\epsilon>0$,
 we say that a vector $g$ is an \emph{$\epsilon$-subgradient} of $f$ at $x\in\dom(f)$ if
 \[
 f(z)\geq f(x) + \langle g, z-x\rangle - \epsilon,\quad \forall z\in\reals^n\enspace.
 \]
The \emph{$\epsilon$-subdifferential} is the set of all $\epsilon$-subgradients of $f$ at $x$, \ie,
\[
\partial_{\epsilon}f(x) := \left\{g\mid f(z)\geq f(x) + \langle g, z-x\rangle - \epsilon,\quad \forall z\in\reals^n\right\}\enspace.
\]
\end{definition}

\paragraph{A family of local sets} The $\epsilon$-subdifferential set provides us with a way to restrict the set within which the
actual gradient of the Moreau envelope at $v$, $\nabla f^\gamma(v)$, resides.

In order to do so, we introduce the following parameterization:
\begin{equation}{\label{eq::epsilon}}
\epsilon^\ast(v;z_1) = \overline{f}^\gamma(v;z_1)-f^\gamma(v)\enspace.
\end{equation}

Let us explain what equation~\eqref{eq::epsilon} represents with the help of Figure~\ref{subfig::epis:1}.
For any $\epsilon>0$ denoting the vertical distance from $\overline{f}^\gamma(v;z_1)$ at $v$,
we can draw the hyperplanes (grey lines) that
are tangent to the epigraph of $\overline{f}^\gamma(z;z_1)$ and generate an outer approximation of
$\mathrm{epi}\;(\overline{f}^\gamma(z;z_1))$. Equation~\eqref{eq::epsilon} defines \emph{the smallest
$\epsilon$-subdifferential set for which we can show that the gradient $\nabla f^\gamma(v)$ resides inside.}

\begin{proposition}{\label{prop::eps_subdiff_1}}
The following statements characterize the set in which $\nabla f^\gamma(v)$ is located.
\begin{enumerate}
\item
 The gradient of the Moreau envelope at a point $v$ is contained in
 the $\epsilon^\ast(v;z_1)$-subdifferential set of the upper-bounding quadratic function
 $\overline{f}^\gamma(z;z_1)$ evaluated at $v$ and constructed around the
 query point $z_1$, \ie,
 \[
  \nabla f^\gamma(v)\in\partial_{\epsilon^\ast(v;z_1)}\overline{f}^\gamma(v;z_1)\enspace.
 \]
 \item
 The $\epsilon^\ast(v;z_1)$-subdifferential set is nonempty, convex and compact.
 \end{enumerate}
\end{proposition}
\begin{proof}
 The proof is deferred to Appendix~\ref{app::1}.
\end{proof}

\begin{figure*}
 \centering
  \begin{tabular}[c]{cc}
  \begin{subfigure}[c]{0.40\textwidth}
    \includegraphics[width=1.00\linewidth]{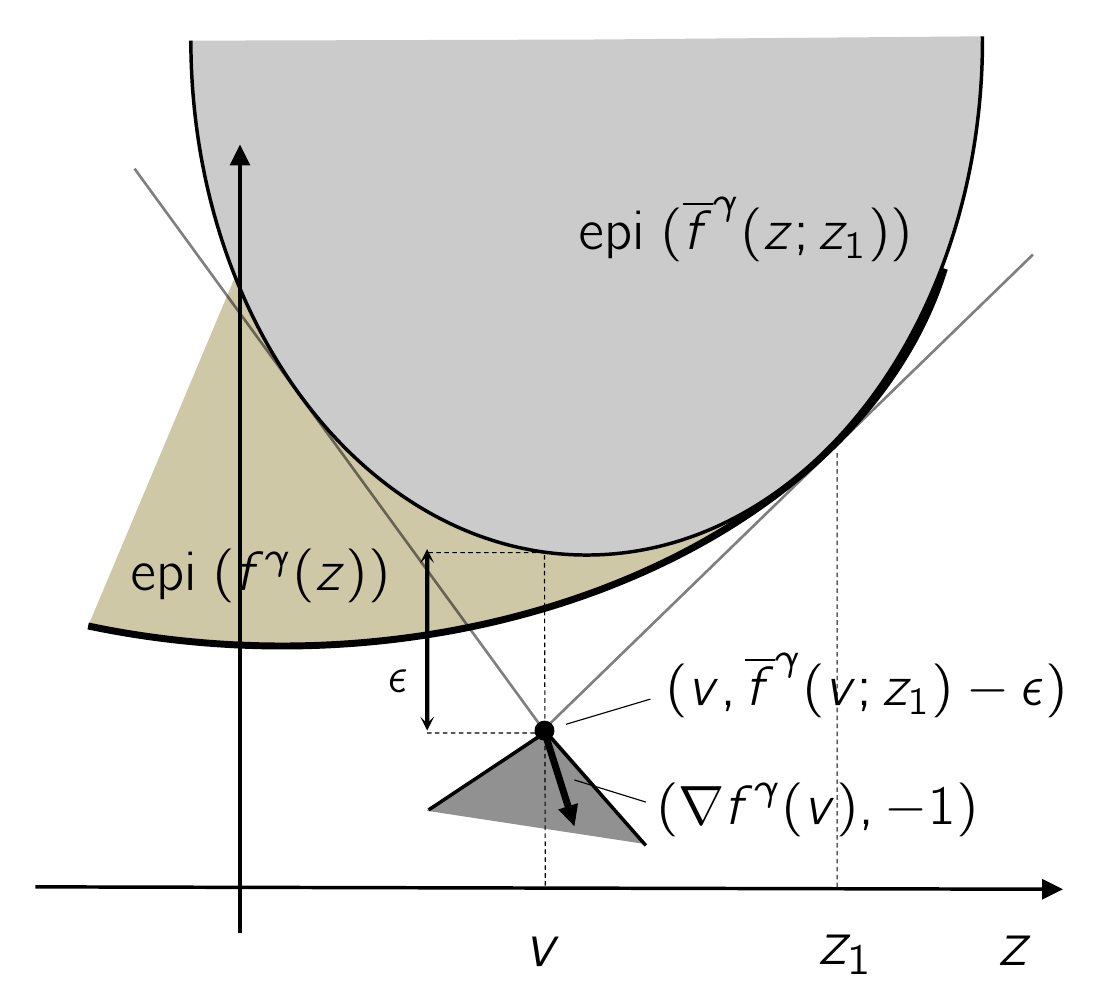}
    \caption{An $\epsilon$-subdifferential set represented as a cone.}
    \label{subfig::epis:1}
  \end{subfigure}&
  \begin{subfigure}[c]{0.40\textwidth}
    \includegraphics[width=1.00\linewidth]{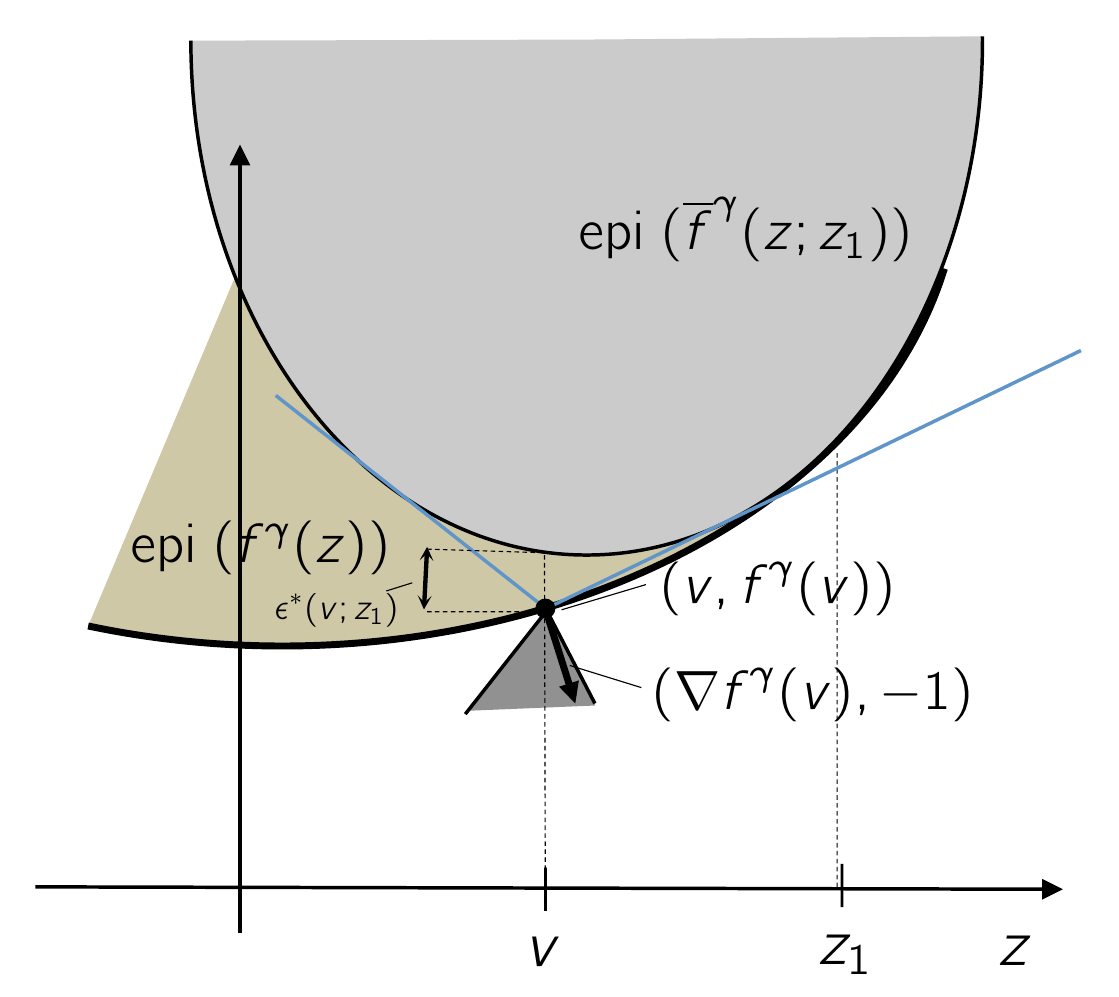}
    \caption{The smallest $\epsilon$-subdifferential set.}
    \label{subfig::epis:2}
  \end{subfigure}\\
\end{tabular}
\caption{The epigraph of the envelope function $f^\gamma(z)$ and that of the quadratic upper bounding function $\overline{f}^\gamma(z;z_1)$ are illustrated. For any point centered at $v$ that lies $\epsilon$ below the epigraph of $\overline{f}^\gamma(z;z_1)$, the $\epsilon$-subdifferential set is depicted as the normal cone of the set constructed by the tangent hyperplanes of $\mathrm{epi}\;(\overline{f}^\gamma(z;z_1))$, depicted with the grey lines, at $(v,\overline{f}^\gamma(v;z_1)-\epsilon)$. The actual gradient $\nabla f^\gamma(v)$ is contained in all $\partial_{\epsilon}\overline{f}^\gamma(v;z_1)$ for $\epsilon\geq \epsilon^\ast(v;z_1)$. As $\epsilon$ is getting smaller, the subdifferential sets shrink, and converge to $\partial_{\epsilon^\ast(v;z_1)}\overline{f}^\gamma(v;z_1)$, depicted by the normal cone to the set illustrated by the blue lines in the right figure.}
\label{fig::sketch1}
\end{figure*}


The proposition suggests the existence of a set that contains the
gradient $\nabla f^\gamma(v)$ , namely the set $\partial_{\epsilon^\ast(v;z_1)}\overline{f}^\gamma(v;z_1)$
which describes a family of admissible slopes.

\paragraph{Explicit representation of the local sets} We will, subsequently, construct $\partial_{\epsilon^\ast(v;z_1)}\overline{f}^\gamma(v;z_1)$
in an explicit way.
To this end, we introduce the function
\begin{equation*}{\label{eq::translate_func}}
\overline{f}_z^\gamma(d;z_1) = \overline{f}^\gamma(z+d;z_1) - \overline{f}^\gamma(z;z_1), \quad \forall d\in\reals^n
\end{equation*}
and its conjugate~\cite[Chapter~23]{Rockafellar70convexanalysis}
\begin{equation}{\label{eq::translate_func_conj}}
(\overline{f}_z^\gamma)^\star(g;z_1) = (\overline{f}^\gamma)^\star(g;z_1) + \overline{f}^\gamma(z;z_1) - \langle z,g \rangle\enspace.
\end{equation}
The following Proposition holds.
\begin{proposition}{\label{prop::subdiff_epi}}
The $\epsilon$-subdifferential set of $\overline{f}^\gamma(z;z_1)$, evaluated at $z=v$, is given by
\begin{equation}{\label{eq::translate_func_epi}}
\partial_{\epsilon}\overline{f}^\gamma(v;z_1) = \left\{ g\;\mid\;(\overline{f}_v^\gamma)^\star(g;z_1) \leq \epsilon\right\}\enspace,
\end{equation}
for any $\epsilon>0$.
\end{proposition}

Proposition~\ref{prop::subdiff_epi} relates the $\epsilon$-subdifferential of a function with the epigraph of its conjugate.
More detailed discussions can be found in~\cite[Chapter~23]{Rockafellar70convexanalysis}, \cite[Section~6.7.1]{bertsekas_cvx_algos}. This equivalence allows us to
construct the $\epsilon^\ast(v;z_1)$-subdifferential set described in Proposition~\ref{prop::eps_subdiff_1}
provided that~\eqref{eq::translate_func_conj} is computable. We prove below that this is indeed the case.

\begin{proposition}{\label{prop::eps_subdiff_2}}
The set $\partial_{\epsilon}\overline{f}^\gamma(v;z_1)$ is given by
\begin{align}{\label{eq::eps_subdiff_2}}
\partial_{\epsilon}\overline{f}^\gamma(v;z_1) = \nonumber&\Big\{g\;\mid\;(\gamma/2)\|g-\nabla f^\gamma(z_1)\|_2^2 \\
                      \nonumber&- \langle g-\nabla f^\gamma(z_1),v-z_1\rangle \\
                      &+ (1/2\gamma)\|z_1-v\|_2^2\leq \epsilon\Big\}\enspace.
\end{align}
\end{proposition}
\begin{proof}
 The conjugate of~\eqref{eq::quadratic_bound} can be easily computed using the definition since the function is a convex quadratic.
 The claim follows directly from computing it and substituting the result in~\eqref{eq::translate_func_conj}.
\end{proof}
\paragraph{The smallest local set} Propositions~\ref{prop::eps_subdiff_1} and~\ref{prop::eps_subdiff_2} tell us that
the gradient $\nabla f^\gamma(v)$ is contained within any of the 2-norm balls associated to the query point $z_1$ and described by~\eqref{eq::eps_subdiff_2}.
The radius of the balls varies with $\epsilon$. A question that naturally arises is
related to the size of this set, namely, \emph{whether we can find the smallest set in the family described by~\eqref{eq::eps_subdiff_2}.}
\begin{proposition}{\label{prop::eps_subdiff_smallest}}
Given a point $v\in\reals^n$ and a distance $\epsilon<\epsilon^\ast(v;z_1)$, there exists a function $f(x)$ with an associated Moreau envelope $f^\gamma(z)$ and
the corresponding local upper-bounding function $\overline{f}^\gamma(v;z_1)$ centered at some query point $z_1$, such that
$\nabla f^\gamma(v)\notin\partial_{\epsilon}\overline{f}^\gamma(v;z_1)$.
\end{proposition}
\begin{proof}
 The proof can be found in Appendix~\ref{app::2}.
\end{proof}
Proposition~\ref{prop::eps_subdiff_smallest} states that plugging $\epsilon^\ast(v;z_1)$ in~\eqref{eq::eps_subdiff_2}
results in the smallest set of the family containing $\nabla f^\gamma(v)$ for a general function $f$. The set can be written as
\begin{align}{\label{eq::eps_subdiff_explicit}}
&\mathcal{G}(v;z_1) \nonumber := \partial_{\epsilon^\ast(v;z_1)}\overline{f}^\gamma(v;z_1) \\
&  = \scalebox{1.00}{$\left\{g\;\mid\;\frac{\gamma}{2}\|g-\nabla f^\gamma(z_1)\|_2^2 - \langle g,v-z_1\rangle\leq f^\gamma(z_1)-f^\gamma(v)\right\}$}.
\end{align}

\paragraph{A global gradient set} To recap, we showed that $\nabla f^\gamma(v)$ is contained in the family of $\epsilon$-subdifferential sets $\partial_{\epsilon}\overline{f}^\gamma(v;z_1)$
constructed around $v$ and associated to the query point $z_1$ (Proposition~\ref{prop::eps_subdiff_1}),
that these sets can be explicitly described (Proposition~\ref{prop::eps_subdiff_2}) and that the smallest set of the family
is recovered for the choice of $\epsilon=\epsilon^\ast(v;z_1)$ from~\eqref{eq::epsilon} (Proposition~\ref{prop::eps_subdiff_smallest}).

Making use of~\eqref{eq::eps_subdiff_explicit}, the set of possible gradients of $f^\gamma$ at $v$ can be described by
\begin{align}{\label{eq::eps_subdiff_all}}
\mathcal{G}(v) \nonumber&:= \bigcap_{j\in\mathcal{J}}\partial_{\epsilon^\ast(v;z_j)}\overline{f}^\gamma(v;z_j)\\
               \nonumber&= \scalebox{1.00}{$\bigcap_{j\in\mathcal{J}}\Big\{g\;\mid\;\frac{\gamma}{2}\|g-\nabla f^\gamma(z_1)\|_2^2 - \langle g,v-z_1\rangle\leq$} \\
                        & \qquad\qquad\qquad\scalebox{1.00}{$f^\gamma(z_1)-f^\gamma(v)\Big\},$}
\end{align}
where $\mathcal{J}$ is a set of indices cooresponding to generated query points. As a consequence,
the set of potential gradients can be described as an intersection of a finite number of 2-norm balls.
It also trivially holds that $\bigcap_{j=1}^{\infty}\mathcal{G}(v;z_j) = \{\nabla f^\gamma(v)\}$, \ie,
upon complete construction of the envelope, the above intersection reduces to the singleton set $\{\nabla f^\gamma(v)\}$.

\paragraph{Relation to tightest upper bounding function} As a last step, we will derive the relation between \emph{$\mathcal{G}(v)$ and the
$\epsilon$-subdifferential set associated to the tightest constructed upper bound of $f^\gamma$ at $v$.}

To this end, we introduce the function corresponding to the convex hull of the generated upper bounding quadratics
\begin{equation}{\label{eq::convex_hull}}
\overline{f}^\gamma(z\mid\mathcal{J}) = \conv\{\overline{f}^\gamma(z;z_j)\mid j\in\mathcal{J}\}\enspace.
\end{equation}
Note the abuse of notation in~\eqref{eq::convex_hull} since it actually describes a set and not a function. To
define it properly, we should write
\[
\overline{f}^\gamma(z\mid\mathcal{J}) = \underset{\substack{\sum_{j\in\mathcal{J}}\theta_j z_j=z,\\ \theta_j\geq0,\;\sum_{j\in\mathcal{J}}\theta_j=1}}{\inf}\sum_{j\in\mathcal{J}}\theta_j\overline{f}^\gamma(z;z_j)\enspace,
\]
according to~\cite[Theorem~5.6]{Rockafellar70convexanalysis}, however we will keep using~\eqref{eq::convex_hull} for simplicity.

The function is illustrated in Figure~\ref{fig::sketch2} for two scalar functions. The dependence on $\mathcal{J}$ demonstrates
that the function is refined (becomes tighter) with the generation of new query points. The following result holds:
\begin{theorem}{\label{thm::main}}
Given a number of generated query points $\{z_j\},\;j\in\mathcal{J}$ and $v\in\reals^n$,
\[
\mathcal{G}(v)=\partial_{\epsilon^\ast(v\mid\mathcal{J})}\overline{f}^\gamma(v\mid\mathcal{J})\enspace,
\]
\end{theorem}
where $\epsilon^\ast(v\mid\mathcal{J})=\overline{f}^\gamma(v\mid\mathcal{J})-f^\gamma(v)$.
\begin{proof}
The proof can be found in Appendix~\ref{app::3}.
\end{proof}

\begin{figure*}
 \centering
  \begin{tabular}[c]{cc}
  \begin{subfigure}[c]{0.45\textwidth}
    \includegraphics[width=1.00\linewidth]{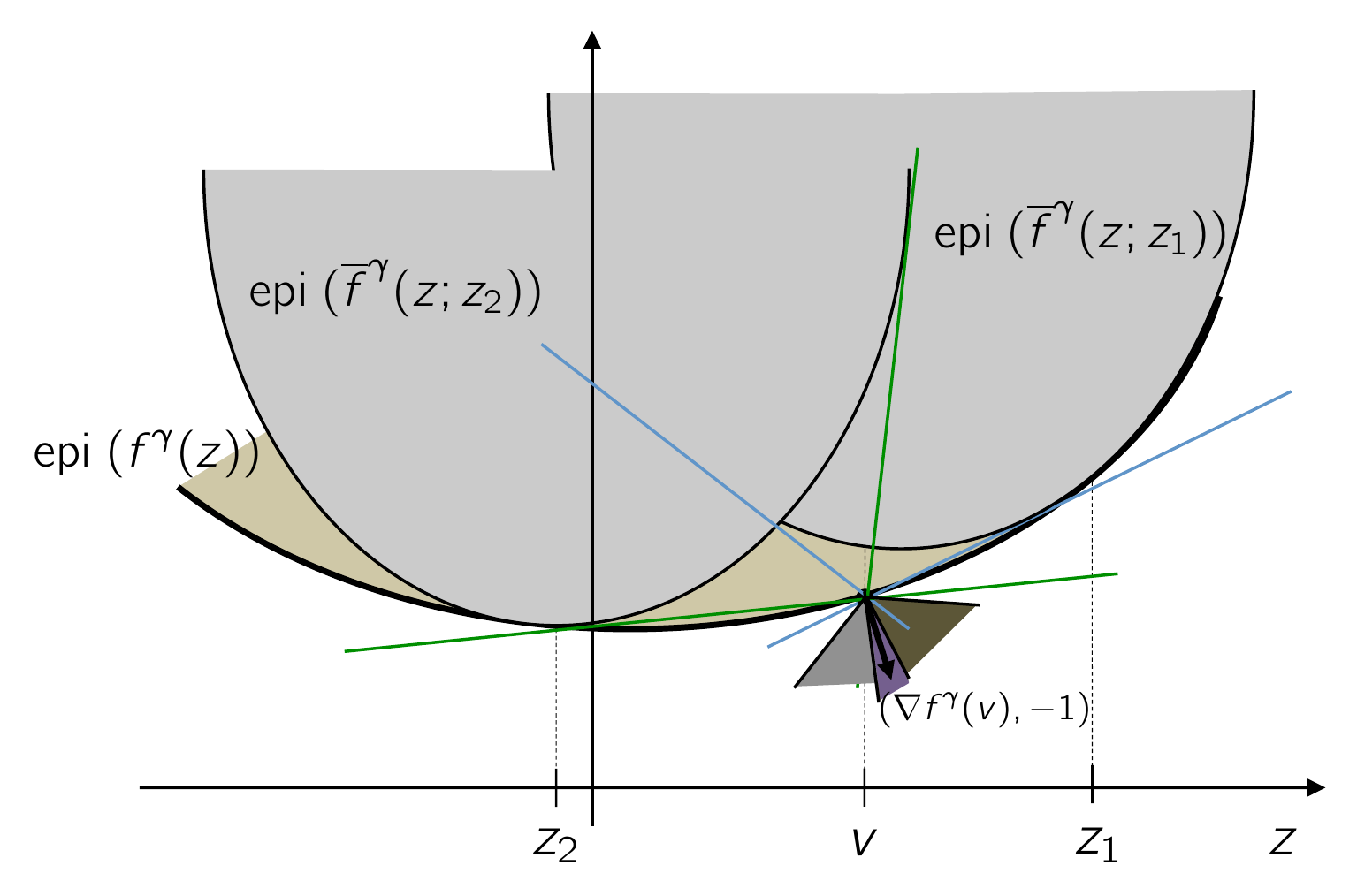}
    \label{subfig::epis:3}
  \end{subfigure}&
  \begin{subfigure}[c]{0.45\textwidth}
    \includegraphics[width=1.00\linewidth]{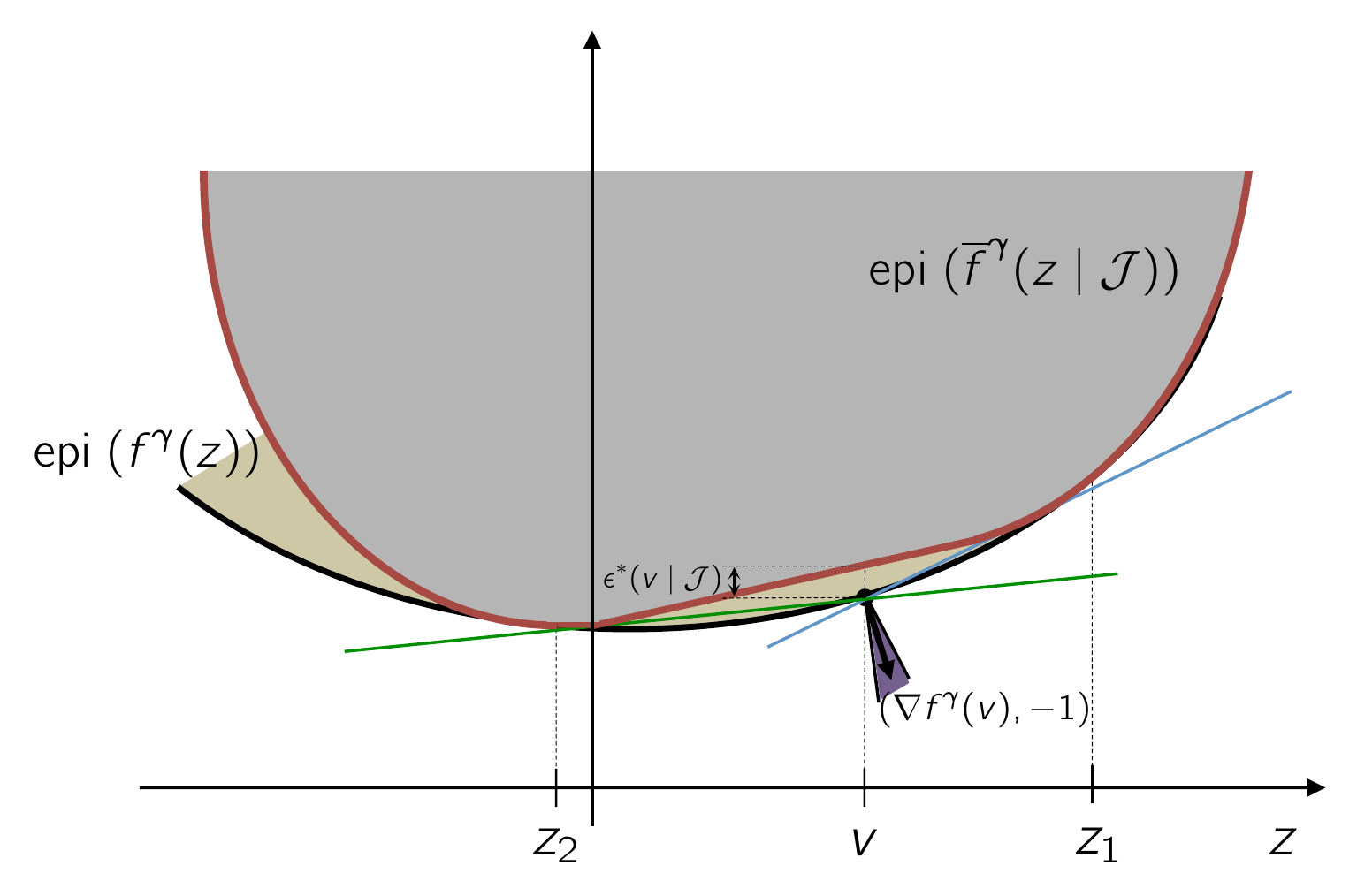}
    \label{subfig::epis:4}
  \end{subfigure}\\
\end{tabular}
\caption{Assuming two query points $z_1$ and $z_2$, $\nabla f^\gamma(v)$ can be located in the intersection of the two normal cones, which is sketched in purple color.
In addition, the function $\overline{f}^\gamma(z\mid\mathcal{J})=\conv\{\overline{f}^\gamma(z;z_1),\overline{f}^\gamma(z;z_2)\}$ is drawn in deep red color. It is evident that the normal cone corresponding to the intersection of the two sets (left) is identical to the set $\partial_{\epsilon^\ast(v\mid\mathcal{J})}\overline{f}^\gamma(v\mid\mathcal{J})$ on the right.}
\label{fig::sketch2}
\end{figure*}

\paragraph{An approximation of the global set} A problematic aspect of the global set $\mathcal{G}(v)$ when it comes to computing it, is the dependence on the unknown quantity $f^\gamma(v)$ in~\eqref{eq::eps_subdiff_explicit}. We want to drop this dependency by approximating $\mathcal{G}(v)$ with another set, namely $\hat{\mathcal{G}}(v)$.

To that end, we use the fundamental inequality (see, \eg,~\cite[Proposition~6.1.9]{bertsekas_cvx_algos}) 
\[
\scalebox{0.88}{$f^\gamma(z_1)-f^\gamma(v) \leq -\langle\nabla f^\gamma(z_1),v-z_1\rangle - \frac{\gamma}{2}\|\nabla f^\gamma(v)-\nabla f^\gamma(z_1)\|^2_2$}.
\]
Using the inequality above in the right-hand side of~\eqref{eq::eps_subdiff_explicit} leaves us, still, with another unknown quantity, \ie, $\nabla f^\gamma(v)$.
A reasonable thought would be to replace $\nabla f^\gamma(v)$ with $g$ and look for a $g\approx\nabla f^\gamma(v)$ that satisfies the above inequality. By doing so, we end up with the inequality
\begin{align}{\label{eq::eps_subdiff_cocoercive}}
&\hat{\mathcal{G}}(v;z_1) := \scalebox{0.88}{$\left\{g\;\mid\;\gamma\|g-\nabla f^\gamma(z_1)\|_2^2 - \langle g-\nabla f^\gamma(z_1),v-z_1\rangle\leq 0\right\}$}.
\end{align}
Interestingly,~\eqref{eq::eps_subdiff_cocoercive} is the necessary condition
for convexity and Lipschitz continuity of the gradient of $f^\gamma$, the so-called
\emph{co-coercivity property} of the gradient, applied to the pair of points $(z_1,v)$. 
As such, the set $\hat{\mathcal{G}}(v;z_1)$ is a \emph{valid approximation of~\eqref{eq::eps_subdiff_explicit}} in the sense that it contains the unknown gradient $\nabla f^\gamma(v)$. In addition, it does not depend 
on any unknown quantities. The (approximate) global set is given by 
\[
\hat{\mathcal{G}}(v)=\bigcap_{j\in\mathcal{J}}\hat{\mathcal{G}}(v;z_j)\enspace.
\]
\begin{remark}
We cannot say anything about the relative size of the sets $\hat{\mathcal{G}}(v)$ and $\mathcal{G}(v)$. 
As a matter of fact, we can devise examples that show that no set is strictly contained in the other. 
Therefore, we simply use $\hat{\mathcal{G}}(v)$ in the computations since we can explicitly compute it.
\end{remark}

\section{Communication}{\label{sec::communication}}
Let us now move back to the original distributed setting. 

In~\cite{learnProx_ECC} we considered a special case of~\eqref{eq::distopt}, namely when $f_i(x_i)=\delta(x_i\mid\mathcal{X}_i)$ and $h$ is differentiable with $L$-Lipschitz continuous gradient.
In this particular setting,~\eqref{eq::distopt} can be solved using the Projected Gradient Method.
Although useful in many cases, the projected gradient method's applicability is limited
to problems where the objective is minimized in the intersection of a number of convex sets.
In this work we consider the more general case of solving~\eqref{eq::distopt} when $h$ is not necessarily
differentiable and $f_i$ encapsulates a convex objective restricted within some
convex set. The Alternating Direction Method of Multipliers (ADMM) is a widely
applicable and resilient method to consider for such problems.

The method comprises the iterations:
\begin{equation}   \label{eq::admm}
  \begin{array}{ll}
    x^{k+1}       &= \underset{x\in\reals^n}{\argmin}\left\{h(x) + (\rho/2)\|x-y^k+(1/\rho)\lambda^k\|^2\right\} \\
    \tilde{x}^{k+1} &= \eta^kx^{k+1} + (1-\eta^k)y^k \\
    y_i^{k+1}     &= \underset{y_i\in\reals^{n_i}}{\argmin}\left\{f_i(y_i) + (\rho/2)\|y_i-\tilde{x}_i^{k+1}-(1/\rho)\lambda_i^k\|^2\right\} \\
    \lambda^{k+1} &= \lambda^{k} + \rho(\tilde{x}^{k+1}-y^{k+1})\enspace,
  \end{array}
\end{equation}
$\rho>0$ and $\eta^k\in(0,2)$ a relaxation parameter.
The splitting described in~\eqref{eq::admm} is achieved by introducing a copy $y=x$ and the associated
dual variables $\lambda$. Observe that the $y$-minimization step constitutes local proximal
minimization problems of the form
\begin{align}{\label{eq::prox2MoreauGrad}}
y_i^{k+1} \nonumber&= \prox{_{\frac{1}{\rho} f_i}}(\underbrace{\tilde{x}_i^{k+1}+(1/\rho)\lambda_i^k}_{v_i^k}) \\
          &= v_i^k - \gamma\nabla f_i^\gamma(v_i^k),\quad \gamma=1/\rho\enspace.
\end{align}

Our purpose is to solve~\eqref{eq::distopt} using~\eqref{eq::admm} while minimizing communication between the
coordinator and the agents. With this purpose in mind, we use the theory developed in the previous
section to build sets in order to locate the solutions of the resulting proximal minimization problems that
are repetitively solved in every iteration of the algorithm. Let us denote with
$\hat{\mathcal{G}}_i(v_i;z_{i,j}),\;i=1,\ldots,N$ the set given by~\eqref{eq::eps_subdiff_cocoercive}, associated
to agent $i$. The coordinator can keep in its memory a set
$\hat{\mathcal{G}}_i(v_i):=\bigcap_{j\in\mathcal{J}_i}\hat{\mathcal{G}}_i(v_i;z_{i,j})$, $\mathcal{J}_i$ being the number of queries generated up to
the current algorithmic iteration, for every agent $i=1,\ldots,N$.
The set is updated whenever a communication round occurs.
With every query point, $\hat{\mathcal{G}}_i(v_i)$ is augmented by one more inequality. \emph{The challenge that needs to
be addressed is the design of a certification test, based on which the decision to
communicate or not is made.}

We first need a scheme that selects a valid approximation of $\nabla f_i^\gamma(v_i^k)$, denoted as $g_i^k$. We know that any point lying in the interior of the
set $\hat{\mathcal{G}}_i(v_i^k)$ is a potential candidate. A reasonable guess for a `good' $g_i^k$ would be to
identify the hyperplane that is closest to $f_i^\gamma$ at $v$, namely to find the index
\begin{equation}{\label{evaluation}}
 j^{\ast} = \underset{j\in\mathcal{J}_i}{\argmax}\left\{f_i^\gamma(z_{i,j}) + \langle \nabla f_i^\gamma(z_{i,j}), v_i^k-z_{i,j}\rangle\right\}\enspace,
\end{equation}
expecting that the gradient of the envelope won't have changed much if the points $z_{i,j^{\ast}}$ and $v_i^k$ are close to each other.
If $\nabla f^\gamma(z_{i,j^{\ast}})\in\hat{\mathcal{G}}_i(v_i^k)$, then $g_i^k=\nabla f^\gamma(z_{i,j^{\ast}})$, otherwise the projection
of $\nabla f^\gamma(z_{i,j^{\ast}})$ onto $\hat{\mathcal{G}}_i(v_i^k)$ is taken. The latter is achieved by solving the quadratically
constrained quadratic program (QCQP)
\begin{equation}   \label{eq::qcqp}
  \begin{array}{ll}
    \mbox{minimize} & \|g_i-\nabla f_i^\gamma(z_{i,j^{\ast}})\|_2  \\
    \mbox{subject to} & g_i \in \hat{\mathcal{G}}_i(v_i^k;z_{i,j}),\;j=1,\ldots,\mathcal{J}_i\enspace,
  \end{array}
\end{equation}
with variable $g_i\in\reals^{n_i}$.

Consequently, we need to decide whether the resulting $g_{i}^{k}$ checks out as
a valid gradient estimate. In order to do so, we propose a heuristic test based on 
the progress of the ADMM residual. The residual is defined with respect to the 
consensus condition $x=y$, \ie, $s^k=\|x^k-y^k\|$, and the communication test reads:
\begin{equation}\label{eq::VI}
\mathcal{T}(s^k,s^{k+1}) := s^k-s^{k+1} < 0\enspace.
\end{equation}

Monotonic decrease of the residual is neither a necessary nor a sufficient condition for the 
convergence of ADMM. It can be used, nevertheless, as a certification test for assessing the validity of the
approximate proximal minimizer.

More specifically, any guess for $\nabla f_i(v_i^k)$ denoted by $g_i^k$ will give rise to
an approximate minimizer $\hat{y}_i^{k+1}=v_i^k-\gamma g_i^k$. This minimizer is used, in turn, to compute the
residual $s_i^{k+1}=x_i^{k+1}-\hat{y}_i^{k+1}$. After performing these updates for all the agents,
inequality~\eqref{eq::VI} is checked. If it holds, no communication is necessary,
while in the opposite case all the agents need to communicate.

The rational behind checking inequality~\eqref{eq::VI} is to assess whether any progress has been
made with regard to reaching agreement between the solutions.
If not, there is an indication that (some of) the approximate optimizers $\hat{y}_i^{k+1}$ were not `in the right direction',
thereby it is better to correct via communication. In addition, the test can be performed at 
the coordinator's level with no extra knowledge about the agents.

Putting together the results of the preceding sections, we describe below the reduced
communication scheme that solves problem~\eqref{eq::distopt}.
\begin{algorithm}
\caption{Distributed Optimization with Estimated Proximal Operator}
\label{al::algo}
\begin{algorithmic}[1]
\Require $x_i^0\in \reals^{n_i}$, $\mathcal{G}_i=\emptyset,\;i=1,\ldots,N$, constant $\rho>0$. Iteration counter is set to $k=0$, $k_{\mathrm{stop}}>0$.
\While{$k<k_{\mathrm{stop}}$}
 \State Compute $x^{k+1}$ from~\eqref{eq::admm} \Comment{Coordinator}
 \State Compute $v^{k} = \tilde{x}^{k+1}+(1/\rho)\lambda^k$ \Comment{Coordinator}
 \State Solve~\eqref{eq::qcqp} $i=1,\ldots,N$ to get $g^k$ \Comment{Coordinator}
 \State Compute $y^{k+1} = v^k - (1/\rho) g^k$ \Comment{Coordinator}
 \State Compute $\lambda^{k+1}$ from~\eqref{eq::admm} \Comment{Coordinator}
 \State Compute $s^{k+1}=\|x^{k+1}-y^{k+1}\|$ \Comment{Coordinator}
  \If{$\mathcal{T}(s^k,s^{k+1}) < 0$}
    \State Transmit $v_{i}^{k}$ to Agent $i,\;i=1,\ldots,N$  \Comment{Coordinator}
    \State Compute $y_i^{k+1}$ from~\eqref{eq::admm} \Comment{Agent $i$}
    \State Compute $\lambda_i^{k+1}$ from~\eqref{eq::admm} \Comment{Agent $i$}
    \State Transmit $(y_i^{k+1},\lambda_{i}^{k+1})$ to the Coordinator \Comment{Agent $i$}
    \State Update sets of query points with $z_{i,j}=v_i^k$ and $\hat{\mathcal{G}}_i(v_i)=\hat{\mathcal{G}}_i(v_i)\cap\hat{\mathcal{G}}_i(v_i;z_{i,j})$
           with $\nabla f_i^\gamma(v_i^k)=\rho(v_i^k-y_i^{k+1})$, $i=1,\ldots,N$ \Comment{Coordinator}
  \EndIf
\EndWhile
\State $k \leftarrow k+1$
\end{algorithmic}
\end{algorithm}

\begin{remark}
The test devised in~\eqref{eq::qcqp} is by no means the only way to find an
approximate gradient. As a matter of fact, it is proposed in~\cite{learnProx_ECC} to
use the Chebyshev center of the set $\mathcal{G}_i(v_i)$ as the gradient estimate. The
latter would, however, result in a semidefinite program (SDP), which is generally more
expensive to solve that the proposed QCQP.
\end{remark}

\begin{remark}
In the way Algorithm~\ref{al::algo} is written, if the test is not passed, all the agents
will communicate (Line~8). One could, however, design different communication schemes. For
example, the coordinator could start transmitting to the agents sequentially until the test
is passed. Indeed, any subset of agents could be selected based on some
specified criterion, as long as the test is passed. We deliberately avoided to comment on
these potential formulations for the sake of brevity.
\end{remark}

\section{Convergence}{\label{sec::convergence}}
In this section we give conditions that ensure convergence of the scheme described by
Algorithm~\ref{al::algo}. The analysis comprises two steps: We \emph{first cast~Algorithm~\ref{al::algo}
as an ADMM iteration with additive error in its $y$-update.} Consequently,
\emph{we employ known results from inexact fixed point iteration theory to give
conditions for convergence of that `ADMM iteration with errors'.}

For analysis purposes, we are going to express ADMM as a \emph{Douglas-Rachford (DR) iteration of the dual
variables $\lambda$}. The equivalence of the two algorithms can be found in, \eg,~\cite{Eckstein_dr},
while a detailed derivation of ADMM from DR is given in the lecure notes~\cite{EE_236C}. By using standard
conjugate duality, we can rewrite~\eqref{eq::distopt} as
 \begin{equation}{\label{eq::dual}}
  \begin{aligned}
  &{\text{minimize}} && H(\lambda) + F(\lambda) \enspace,
  \end{aligned}
 \end{equation}
  where $\lambda$ is the dual variable in~\eqref{eq::admm}, $H(\lambda):=h^\star(-\lambda)$
  and $F(\lambda):=\sum_{i=1}^N f_i^\star(\lambda_i)$.
Applying the ADMM iteration~\eqref{eq::admm} to solve~\eqref{eq::distopt}
is equivalent to applying the following DR iteration to problem~\eqref{eq::dual}:
\begin{align}   \label{eq::dr}
    r^{k+1}       \nonumber&= \prox{_{\rho H}}(2\lambda^k-z^k) \\
    z^{k+1}       \nonumber&= z^k + \eta^k(r^{k+1} - \lambda^k) \\
    \lambda^{k+1} &= \prox{_{\rho F}}(z^{k+1})\enspace,
\end{align}
where
\[
\prox{_{\rho F}}(z^{k+1}) := (\prox{_{\rho f_1^\ast}}(z_1^{k+1}),\ldots,\prox{_{\rho f_N^\ast}}(z_N^{k+1})).
\]

We will now illustrate that the approximate minimizer $y_i^{k+1}$ of~\eqref{eq::admm} can be expressed
as the solution of a proximal minimization problem with some additive error.
Let us first express the approximate iteration
$\hat{y}_i^{k+1} = v_{i}^{k} - \gamma g_{i}^{k}$ as \emph{an inexact solution to a proximal minimization problem}.
To this end, we introduce the error sequence
\begin{equation}{\label{eq::error_sequence}}
\{e^k\}=\{g^k-\nabla f^\gamma(v^k)\}\enspace,
\end{equation}
with $\nabla f^\gamma(v^k) = (\nabla f_1^\gamma(v_1^k),\ldots,\nabla f_N^\gamma(v_N^k))$.
Using~\eqref{eq::error_sequence}, the following result holds:
\begin{proposition}\label{prop::inexact_DR}
Algorithm~\ref{al::algo} can be expressed as the inexact DR iteration
\begin{align}   \label{eq::dr_inexact}
    r^{k+1}       \nonumber&= \prox{_{\rho H}}(2\lambda^k-z^k) \\
    z^{k+1}       \nonumber&= z^k + \eta^k(r^{k+1} - \lambda^k) \\
    \lambda^{k+1} &= \prox{_{\rho F}}(z^{k+1}) + e^k\enspace,
\end{align}
with $e^k$ defined in~\eqref{eq::error_sequence}.
\end{proposition}
\begin{proof}
The proof is deferred to Appendix~\ref{app::4}.
\end{proof}

Capitalizing on this result, the inexact DR iteration~\eqref{eq::dr_inexact}
can be analyzed in the more general context of inexact fixed point iterations,
analyzed in detail in~\cite[Section~5.2]{Liang2016}. To this end, the iteration
can be written in the more compact notation
\begin{equation}{\label{eq::fixed_point}}
z^{k+1} = z^k + \eta^k(Tz^k+\zeta^k-z^k)\enspace,
\end{equation}
by introducing the operator
\[
T = \frac{1}{2}\left(\refl{_{\rho H}}\refl{_{\rho F}}+I\right),
\]
with $T:\reals^n\mapsto\reals^n$ and
\begin{equation}{\label{eq::global_error}}
 \zeta^k = \frac{1}{2}\left(\refl{_{\rho H}}\left(\refl{_{\rho F}}(z^k)+2e^k\right)+z^k\right)-Tz^k\enspace.
\end{equation}
With iteration~\eqref{eq::fixed_point} in place, we are ready to give conditions for the convergence of
the DR iteration~\eqref{eq::dr_inexact} and, consequently, of Algorithm~\ref{al::algo}.
\begin{theorem}\cite[Theorem~2]{combettes:hal-00621820}{\label{thm::convergence}}
The multiplier sequence $\{\lambda^k\}$ generated from Algorithm~\ref{al::algo} will
converge to an optimizer of~\eqref{eq::dual} provided that $\sum_{k=1}^{\infty}\eta^k\|e^k\|<\infty$ and 
$\sum_{k=1}^{\infty}\eta^k(2-\eta^k)=\infty$.
\end{theorem}
\begin{remark}
The sequence $\{\|e^k\|\}$ is bounded since $\|e^k\| \leq N\underset{1\leq i\leq N}{\max}\|g_i^k-\nabla f^{\gamma}_i(v_i^k)\|$. Since condition~\eqref{eq::VI} will be enforcing communication until feasibility is achieved, more query points will be generated and $\lim\limits_{j\rightarrow\infty}\bigcap_{j\in\mathcal{J}_i}\hat{\mathcal{G}}_i(v_i;z_{i,j})=\{\nabla f^{\gamma}_i(v_i)\}$. Since we do not know anything about the rate of convergence, we need to resort to a generic sequence
$\{\eta^k\}$ that would enforce convergence, \eg, $\eta^k=1/(k^2\|e^k\|)$. The quantity $\|e^k\|$ can be (over-)approximated by taking, \eg, the smallest ellipsoid of the ones that are intersected and compute its diameter.
\end{remark}
\begin{remark}
The convergence result presented above is only useful for analysis purposes, since in practice we would not
use a diminishing stepsize which typically results in very slow convergence.
\end{remark}

\section{Application: Optimal Exchange}{\label{sec::application}}
We consider the problem of cooperative tracking of a reference signal from a population of buildings.
The specific problem of interest, first introduced in~\cite{EPFL_Paolone}, is that of the so-called dispatchability of distribution feeders where the main target is to achieve virtually perfect dispatchability of a set of devices consisting of uncontrollable loads and distributed generation. The idea comprises two stages, namely a day-ahead prediction of a dispatch plan and real-time operation, broken further into a high-level and a faster low-level controller, both model-based. Our interest is to solve the high-level Model Predictive Control (MPC) problem that takes the form
\begin{equation}{\label{eq::Dispatch}}
\begin{aligned}
&{\text{minimize}}   && \sum_{i=1}^Nf^{\mathrm{cb}}_i(p^{\mathrm{cb}}_i,u_i,x_i)\\
&{\text{subject to}} && \sum_{t=1}^T\sum_{i=1}^N(p_i^\mathrm{cb}(t)-\hat{p}_i^\mathrm{cb}(t))=r(t)\enspace.
\end{aligned}
\end{equation}
The variable $p^{\mathrm{cb}}_i(t)$ refers to the total amount (electrical equivalent) of the thermal consumption of the $i^{\mathrm{th}}$ controllable building (CB) at time instant $t$, where $i=1,\ldots,N$, and $p^{\mathrm{cb}}_i=(p^{\mathrm{cb}}_i(0),\ldots,p^{\mathrm{cb}}_i(T-1))\in\reals^T$, $p^{\mathrm{cb}}=(p^{\mathrm{cb}}_1,\ldots,p^{\mathrm{cb}}_N)\in\reals^{NT}$. The reference power profile is denoted by $r(t)$, while time spans from $t=0,\ldots,T-1$, \ie, we have a $T$-timesteps ahead prediction of the power profile. The buildings can participate in the ancillary service market by increasing or decreasing their consumption with respect to some baseline power profile $\hat{p}_i^\mathrm{cb}$.
Finally, the equality constraint enforces that the total power contribution equates the reference power profile. The individual components $f_i(p^{\mathrm{cb}}_i,u_i,x_i)$ are local performance criteria coupled with implicit descriptions of convex sets, constructed by the intersection of linear equations (agent dynamics) and constraints, details that are hidden from the global node. More details are given below.

\subsection{Modeling of the agents}
The microgrid comprises small and medium office buildings, generated by the OpenBuild software~\cite{Gorecki2015OpenBuildA}. The buildings are described as linear dynamical systems, the input to which is the thermal heat $(kW)$ that is entering or leaving each zone, while the output is the temperature at each zone $({}^\circ C)$. The energy conversion systems (electrical to thermal) are modeled as a static map, which is represented by a constant coefficient of performance (COP). An individual building seeks to contribute to the tracking objective while respecting temperature as well as operational constraints. An instance of a local optimization problem for building $i$ is
\begin{align}{\label{eq::BuildingProgram1}}
f_i(p^{\mathrm{cb}}_i,u_i,x_i) = w\sum_{t=0}^{T-1}\max\Big\{\nonumber& c_i(t)(p^{\mathrm{cb}}_i(t)-\hat{p}^{\mathrm{cb}}_i(t)),\\
         \nonumber& 0.5c_i(t)(p^{\mathrm{cb}}_i(t)-\hat{p}^{\mathrm{cb}}_i(t)),\\
         \nonumber& -0.5c_i(t)(p^{\mathrm{cb}}_i(t)+3\hat{p}^{\mathrm{cb}}_i(t))\Big\}\\
         & + \delta((p^{\mathrm{cb}}_i,u_i,x_i)\;|\;\mathcal{C}_{\mathrm{build}})\enspace,
\end{align}
\[
\mathcal{C}_{\mathrm{build}} =
\left \{
                  \begin{array}{ll}
                    x_i(t+1) = A_ix_i(t)+B_iu_i(t) \\
                    x_i(0) = x_i^\mathrm{init} \\
                    p^{\mathrm{cb}}_i(t) = \sum_{j=1}^{M_i}u_{ij}(t) \\
                    C_ix_i(t) \in X_i(t)  \\
                    \|u_i(t)\|_\infty \leq u_i^\mathrm{max}
                  \end{array}
\right\}\enspace,
\]
with $x_i\in\reals^{n_iT},u_i=\{u_{ij}\}_{j=1}^{M_i}\in\reals^{M_iT}$, where $x_i(t)$ are states and $M_i$ is the number of zones of the building.
The term $w>0$ weighs the local objective versus the global tracking cost.
The linear mapping of the states $C_ix_i(t)$ corresponds to zone temperatures that are confined within desired limits. In addition to feasibility, the buildings have certain preferences when it comes to tracking a given
power profile. When the profile causes the total power consumption to exceed some threshold, the cost of electricity
increases, while when the total power consumption drops below some threshold the ambient
temperature might become discomforting (although tolerable) for the residents. This discontent is expressed
in~\eqref{eq::BuildingProgram1} by
the piecewise affine function expressed by the $\mathrm{\max}$ term in the objective. The data $c_i(t)$ is the associated electricity cost for bulding $i$ at time $t$ and $\hat{p}_i^\mathrm{cb}$ is the expected
baseline consumption, the deviation from which is penalized or subsidized.

\begin{table}
\centering
\scalebox{0.55}{
\Large
\begin{tabular}{| l l l c |}
\hline
& & &\\
\multicolumn{4}{| l |}{\textbf{Simulation characteristics}} \\
& & &\\
Data & \multicolumn{2}{ l }{ $1^{\mathrm{st}}$ January 2013} & \\
Location & \multicolumn{2}{ l }{ Lausanne} & \\
Time & \multicolumn{2}{ l }{00:00 - 23:00} & \\
Sampling time & \multicolumn{2}{ l }{60} & $\mathrm{min}$ \\
Horizon & \multicolumn{2}{ l }{24} & $\mathrm{-}$ \\
Dimension $\mathcal{Z}_i$ & \multicolumn{2}{ l }{96} & $\mathrm{-}$ \\
& & &\\

\multicolumn{4}{| l |}{\textbf{Buildings}} \\
& & &\\
Desired temperature & \multicolumn{2}{ l }{20} & $\mathrm{^\circ C}$ \\
Minimum temperature (day/night) & \multicolumn{2}{ l }{18/15} & $\mathrm{^\circ C}$ \\
Maximum temperature (day/night) & \multicolumn{2}{ l }{22/25} & $\mathrm{^\circ C}$ \\
Tariff (day/night) & 21.6/12.7 & 13.15/8.3 & $\mathrm{ct./kWh}$ \\
Heat pump $\mathrm{COP}$ & \multicolumn{2}{ l }{3.0} & $\mathrm{-}$ \\
& & &\\

\multicolumn{4}{| l |}{} \\
& \textbf{Small}&\textbf{Medium} &\\
Number of systems & 28 & 22 & $\mathrm{\#}$ \\
Area & 511 & 4982 & $\mathrm{m^2}$\\
Number of states & 15 & 54 & $\mathrm{-}$\\
Number of inputs & 5 & 18 & $\mathrm{-}$\\
Average thermal consumption & 4 & 40 & $\mathrm{W/m^2}$\\
& & &\\
\hline
\end{tabular}}
\caption{Micro-grid case study overview}
\label{table::tab_micro-grid}
\end{table}
Instead of carrying the full building description as in $\mathcal{C}_{\mathrm{build}}$, which would result
in a high-dimensional optimization problem, the authors in~\cite{building2battery} propose a low-dimensional
modeling abstraction of the building as a `thermal battery'. A robust optimization problem is solved to
ensure that the trackable power profiles $p^{\mathrm{cb}}_i(t)=\sum_{j=1}^{M_i}u_{ij}(t)\in\reals^T$ for building $i$
reside inside a convex set, namely $\mathcal{Z}_i$, and satisfy the constraints imposed in~\eqref{eq::BuildingProgram1}.

Using the aforementioned abstraction, the optimization problem associated to building $i$ becomes
\begin{align}{\label{eq::BuildingProgram2}}
f_i(p^{\mathrm{cb}}_i) = w\sum_{t=0}^{T-1}\max\Big\{\nonumber& c_i(t)(p^{\mathrm{cb}}_i(t)-\hat{p}^{\mathrm{cb}}_i(t)),\\
         \nonumber& 0.5c_i(t)(p^{\mathrm{cb}}_i(t)-\hat{p}^{\mathrm{cb}}_i(t)),\\
         \nonumber& -0.5c_i(t)(p^{\mathrm{cb}}_i(t)+3\hat{p}^{\mathrm{cb}}_i(t))\Big\}\\
         & + \delta(p^{\mathrm{cb}}_i\mid\mathcal{Z}_i)\enspace,
\end{align}

As a final step, we perform the simple reformulation suggested in~\cite[Section~7.3.2]{admm_distr_stats}, and the dispatchability problem~\eqref{eq::Dispatch} reduces to the \emph{optimal exchange} problem given below.
\begin{equation}   \label{eq::optimal_exchange}
  \begin{array}{ll}
    z_i^{k+1}       &= \underset{z_i}{\argmin}\left\{f_i(z_i) - \langle z_i,\lambda^k\rangle + \frac{\rho}{2}\|z_i-(z_i^k-\bar{z}^k)\|_2^2\right\}, \\
                    & i=1,\ldots,N+1 \\
    \lambda^{k+1}   &= \lambda^{k} + \rho (\bar{z}^{k+1} - r/N)\enspace,
  \end{array}
\end{equation}
where we have denoted $z_i=p^{\mathrm{cb}}_i,\;i=1,\ldots,N$, $z_{N+1}=p^\mathrm{bess}$, $f_i=f^{\mathrm{cb}}_i,\;i=1,\ldots,N$, $f_{N+1}=f^\mathrm{bess}$ while $\bar{z}=\frac{1}{N+1}\sum_{i=1}^{N+1}z_i$.

We can now solve~\eqref{eq::Dispatch} with $f_i$ given by~\eqref{eq::BuildingProgram2} using Algorithm~\ref{al::algo}.

\subsection{Simulation setup}
Our purpose is to assign one-hour reference tracking to a population that consists of agents described in the previous section. We consider $N=50$ buildings, different electricity tariffs that vary according to the time of the day as well as the size of the load. The tracking term $w$ in~\eqref{eq::BuildingProgram2} is set to $0.05$ in order to prioritize tracking over the agents' local objectives. Other details associated to the simulation are given in Table~\ref{table::tab_micro-grid}.

We solve problem~\eqref{eq::BuildingProgram2} using ADMM~\eqref{eq::admm}, which simplifies to~\eqref{eq::optimal_exchange} for $\eta^k=\eta=1$ with $\rho=20$. A comparison is performed between the exact solution of the problem instance, and the proposed reduced communication approach presented in Algorithm~\ref{al::algo}.

The proximal minimization problems are solved in MATLAB using the YALMIP optimizer~\cite{YALMIP} with the Gurobi solver. For demonstration purposes, the termination condition measures the distance between the iterate and the optimizer, namely $\|z^k-z^\ast\|/\|z^\ast\|\leq10^{-3}$. In a practical implementation an optimizer is, of course, unknown. The algorithm would terminate by means of a communication test, namely, all the agents would be communicated at some algorithmic iteration and the distance to the actual proximal minimizers would be measured. The termination condition could be designed based on the aforementioned distance. 

Figure~\ref{fig::allpoints} depicts the number of iterations versus the number of communication rounds for the residual being computed as $\|z^k-z^\ast\|/\|z^\ast\|$. We observe that Algorithm~\ref{al::algo} achieves approximately 46\% reduction in communication in comparison to the regular ADMM implementation. The number of iterations is, as expected, increased (from 160 to 247). 

The tracking quality of the mix is depicted in Figure~\ref{fig::tracking}. The small buildings' contribution is colored in pink, while that of the medium scale buildings in blue. The reference signal in red has been normalized with respect to the baseline.

The fact that 87 communication rounds were needed to reach the desired accuracy, as depicted in Figure~\ref{fig::allpoints}, translates into solving a QCQP given by equation~\eqref{eq::qcqp} with 87 quadratic constraints per agent. This results in some considerable computational effort. Inspired by \emph{cutting plane methods with discarded hyperplanes} (see~\cite{Kiwiel1983},~\cite[Section~5.3.2]{bertsekas_cvx_algos}), we perform a similar technique in a heuristic sense, where we only keep a fixed number ($M$) of generated query points $z_{i,j}$ for each agent $i$, where now $j$ resides in a subset of $\mathcal{J}$. We keep a fixed number of the \emph{latest} generated gradients $\nabla f_i^\gamma(z_{i,j}),\;i=1,\ldots,N$, arguing that the most recently-generated points give rise to a better (local) polyhedral approximation of $f_i^\gamma$ around the current point $v_i^k$, and subsequently solve~\eqref{eq::qcqp}. The results are summarized in Table~\ref{table::numerics}, where the performance of the (limited-memory) algorithm is compared to the full-memory instance for several choices of $M$. It is observed that even a small number of recent query points is sufficient to achieve that same performance as keeping all the points ($M=5,10$ in the table). When the number of query points is decreased significantly ($M=1$), the number of iterations needed for convergence is increased considerably, leading to an increase in the number of communication rounds. Finally, the average solve time for the problem~\eqref{eq::qcqp} is stated in the last row of the table. The benefit of using a limited-memory approach is evident. It is also observed that there is no significant difference when the number of generated points is small ($M=1,5,10$). The reason is that the bulk of the computation time is attributed to the parsing phase, while the solve time itself is negligible. This observation advocates for the use of a small, but yet adequate, number of points in the formulation.

To sum up, it is observed that when the proximal solution is estimated,
the communication rounds are significantly reduced, by more than $46\%$. This reduction is accompanied by an increase in the number of iterations
needed for convergence. The cause for this increase can be attributed to the excursions that the trajectory takes, most probably caused by the erroneous optimizers that are selected in the early iterations. We also see that limited-memory approaches that originate from discarding query points perform very well in practice, and cause an almost negligible deterioration in the performance.

\begin{table}
\centering
\scalebox{0.55}{
\Large
\begin{tabular}{ |l|l|l|l|l||l| }
\hline
\multicolumn{6}{ |c| }{Optimal Exchange} \\
\hline
  & \multicolumn{4}{|c|}{\textbf{Estimate}} & \textbf{Exact} \\ \hline
 Hyperplanes $M$  & $1$ & $5$ & $10$ & all & -\\ \hline
 \#  iterations & $358$  & $256$ & $254$ & $247$ & $160$\\ \hline
 Communications (per agent) & $163$  & $92$ & $92$ & $87$ & $160$\\ \hline
 \# constraints (per agent) & $1$  & $5$ & $10$ & $87$ & -\\ \hline
 Average computation time (per agent) & $\sim 1\mathrm{s}$  & $\sim 1\mathrm{s}$ & $\sim 1\mathrm{s}$ & $\sim 7\mathrm{s}$ & -\\ \hline
\end{tabular}}
\caption{Numerical comparison among limited-memory and full-memory versions of Algorithm~\ref{al::algo}.}
\label{table::numerics}
\end{table}

\begin{figure}
    \centering
    \includegraphics[scale=0.22]{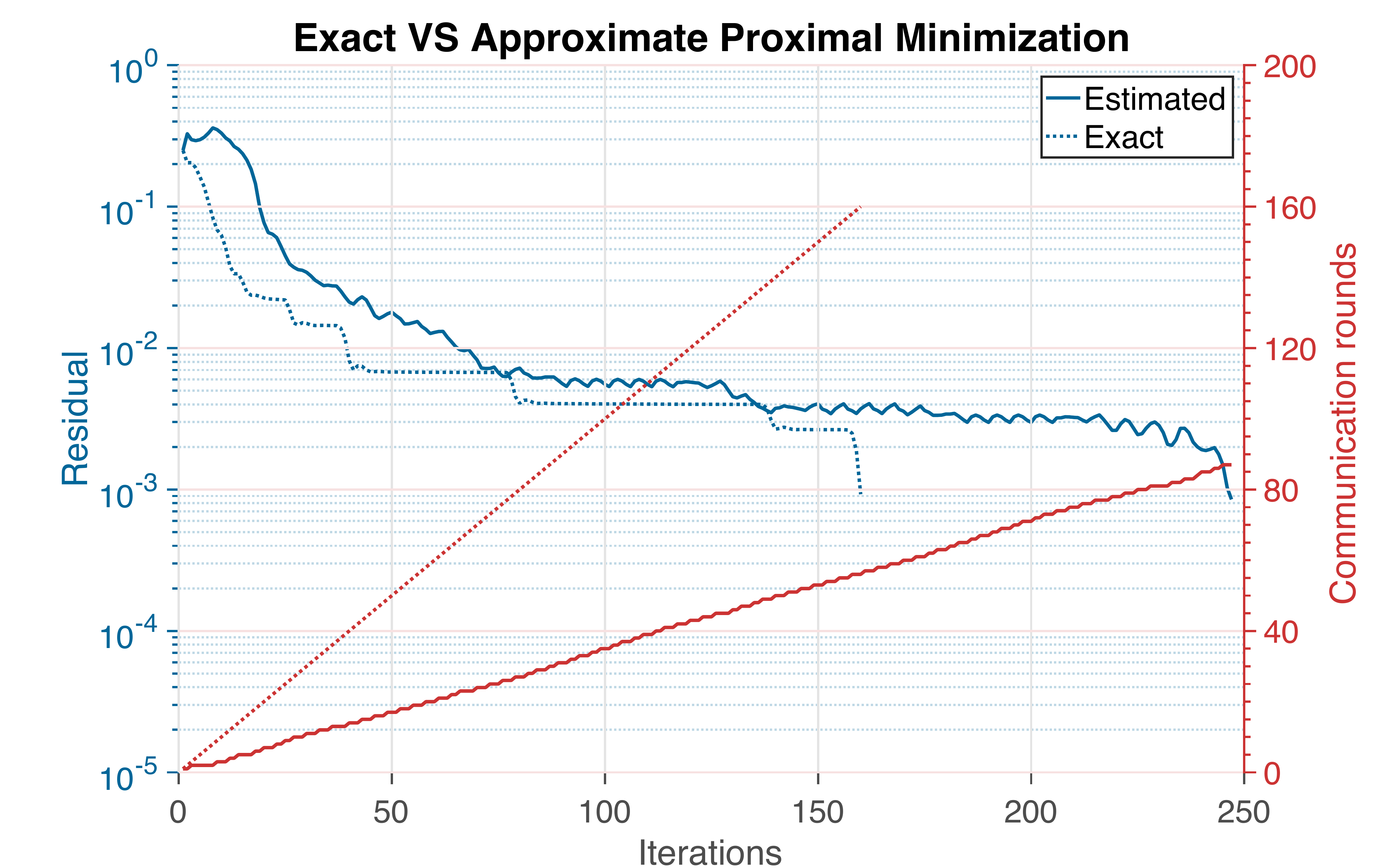}
	\caption{Performance in terms of communication savings.}
	\label{fig::allpoints}
\end{figure}

\begin{figure}
        \centering
        \includegraphics[scale=0.47]{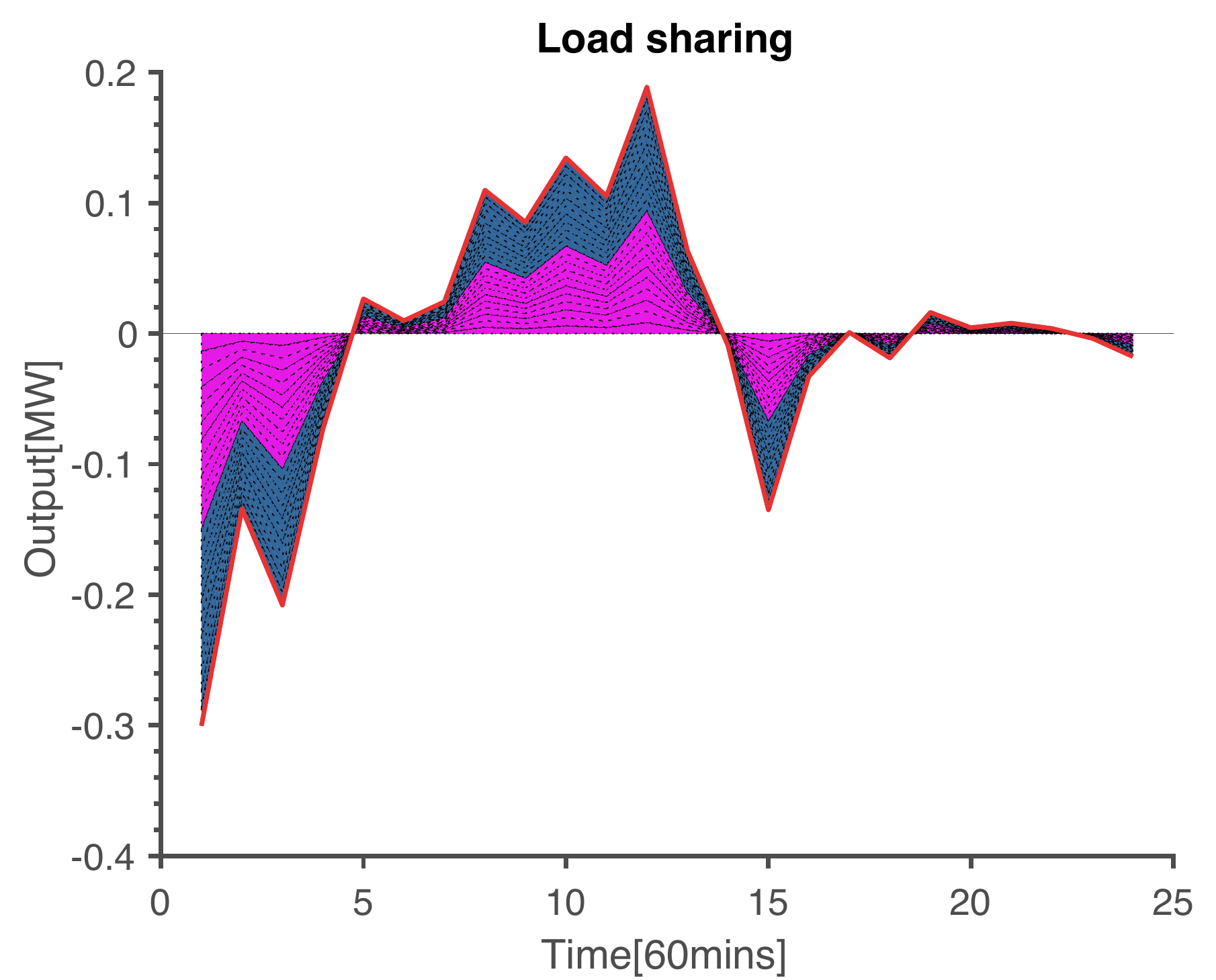}
        \caption{Tracking of a reference by a mix of small (pink) and medium scale (blue) buildings.}
        \label{fig::tracking}
\end{figure}

\section{Conclusion}
Modern multi-agent setups consist of several heterogeneous components equipped with prediction and control algorithms, that attribute to them some decision making capacity. These attributes ask for distributed solutions which come with an inherent communication overhead. We propose a framework where the communication requests are reduced by enabling the central coordinator to gradually `learn' the optimization model of the agents and triggers them based on the result of a predesigned certification test.

A future goal is to generalize the scheme to decentralized settings, for instance, to wireless sensor networks where
only one-hop communication is allowable. Such networks are characterized by high energy consumption related to communication, often several orders of magnitude greater than the energy required for local computations~\cite{1307319}.
As a result, our proposed approach could become of practical relevance.


\bibliographystyle{ieeetr}
\bibliography{IEEEabrv,cones_biblio}

\appendices
\section{Proof of Proposition~\ref{prop::eps_subdiff_1}}{\label{app::1}}
Without loss of generality we derive the analysis for the query point $z_1$ to be consistent with~Figure~\ref{fig::sketch1}.

\begin{enumerate}
\item
We need to show that $\nabla f^\gamma(v)\in\partial_{\epsilon^\ast(v)}\overline{f}^\gamma(v;z_1)$.
Employing Definition~\ref{def::eps_subdiff}, we start by showing that
\[
\overline{f}^\gamma(z;z_1) \geq \overline{f}^\gamma(v;z_1)+\langle \nabla f^\gamma(v), z-v\rangle - \epsilon^\ast(v;z_1),\; \forall z\in\reals^n\enspace.
\]

We will proceed and prove the statement by contradiction. Suppose that  $\exists\;z\in\reals^n$, such that
\[
\overline{f}^\gamma(z;z_1) < \overline{f}^\gamma(v;z_1)+\langle \nabla f^\gamma(v), z-v\rangle - \epsilon^\ast(v;z_1)
\]
\emph{holds}.
From the definition of $\epsilon^\ast(v;z_1)$ (equation~\eqref{eq::epsilon}), the above inequality is equivalent to:
\begin{equation*}
\overline{f}^\gamma(z;z_1) < \langle \nabla f^\gamma(v), z-v\rangle + f^\gamma(v)\enspace.
\end{equation*}
From convexity of $f^\gamma$, it holds that $f^\gamma(v) + \langle \nabla f^\gamma(v), z-v\rangle \leq  f^\gamma(z),\; \forall z\in\reals^n$, therefore
we conclude that $\exists\;z\in\reals^n$ such that
\begin{equation*}
\overline{f}^\gamma(z;z_1) < f^\gamma(z)\enspace,
\end{equation*}
which leads to a contradiction.
Consequently, $\nabla f^\gamma(v)\in\partial_{\epsilon^\ast(v;z_1)}\overline{f}^\gamma(v;z_1)$.

\item The property follows directly from the fact that
 $\overline{f}^\gamma(z;z_1)$ is real-valued, along with~\cite[Proposition~6.7.1]{bertsekas_cvx_algos}.
\end{enumerate}


\section{Proof of Proposition~\ref{prop::eps_subdiff_smallest}}{\label{app::2}}
All we need in order to prove the statement is a function $f$ for which $\epsilon^\ast(v;z_1)$ results in a tight bound.
Let us assume that $f(x)$ is an indicator function of some convex set $\mathcal{X}$ and thus given by
$f(x)=\delta(x\mid\mathcal{X})$. Let us also assume that the set~\eqref{eq::eps_subdiff_2} associated to $\epsilon^\ast(v;z_1)$ is not the smallest,
\ie, $\exists\;\epsilon>0$ such that
$\epsilon = \overline{f}^\gamma(v;z_1) - f^\gamma(v) - \alpha$
for some $\alpha>0$, and that $\nabla f^\gamma(v)\in\partial_{\epsilon}\overline{f}^\gamma(v;z_1)$, or
\[
\overline{f}^\gamma(z;z_1) \geq \overline{f}^\gamma(v;z_1)+\langle \nabla f^\gamma(v), z-v\rangle - \epsilon,\; \forall z\in\reals^n\enspace.
\]
Substituting $\epsilon$ from above, we have that
\begin{equation}{\label{eq::tightest}}
\overline{f}^\gamma(z;z_1) \geq {f}^\gamma(v)+\langle \nabla f^\gamma(v), z-v\rangle + \alpha,\; \forall z\in\reals^n\enspace.
\end{equation}
For $z=z_1$, equation~\eqref{eq::tightest} reads
\[
f^\gamma(z_1) \geq {f}^\gamma(v)+\langle \nabla f^\gamma(v), z_1-v\rangle + \alpha\enspace.
\]
If $z_1\in\mathcal{X}$, \ie, $z_1\in\dom(f)$, it holds $\forall v\in\mathcal{X}$ that
$f^\gamma(z_1)=f^\gamma(v)=0$ and $\nabla f^\gamma(v)=0$. Consequently, the previous inequality results in
$\alpha\leq 0$, which leads to a contradiction. Therefore, there exists no set $\partial_{\epsilon}\overline{f}^\gamma(v;z_1)$
smaller than $\mathcal{G}(v;z_1)$ that contains $\nabla f^\gamma(v)$.

\section{Proof of Theorem~\ref{thm::main}}{\label{app::3}}
We need to compute explicitly the $\epsilon$-subdifferential sets of both $\overline{f}^\gamma(v)$ and $\overline{f}^\gamma(v;z_1)$
and subsequently perform the comparison. We will need the following result from~\cite{Rockafellar70convexanalysis}:
\begin{lemma}\cite[Theorem~16.5]{Rockafellar70convexanalysis}{\label{thm::rocka_convex_hull_conjugate}}
The conjugate of $\overline{f}^\gamma(z\mid\mathcal{J})$ is given by
\[
\scalebox{0.93}{$(\overline{f}^\gamma)^\star(g\mid\mathcal{J}) = (\conv\{\overline{f}^\gamma(z;z_j)\mid j\in\mathcal{J}\})^\star = \underset{j\in\mathcal{J}}{\sup}\{(\overline{f}^\gamma)^\star(g;z_j)\}$}.
\]
\end{lemma}
\begin{itemize}
\item Let us fix the distance of $f^\gamma(v)$ from the quadratic upper bounds as demonstrated in Figure~\ref{fig::sketch2}. The $\epsilon^\ast(v;z_j)$-subdifferential sets for $\overline{f}^\gamma(v;z_j),\;j\in\mathcal{J}$ are:
\begin{align}{\label{eq::thm_proof_2}}
 \partial_{\epsilon^\ast(v;z_j)}\overline{f}^\gamma(v;z_j) = &\nonumber\Big\{g\mid(\overline{f}_v^\gamma)^\star(g;z_j)\leq\epsilon^\ast(v;z_j)\Big\} \\
                                            = &\nonumber\Big\{g\mid(\overline{f}^\gamma)^\star(g;z_j)+\overline{f}^\gamma(v;z_j)\\
                                            &\qquad-\langle v,g\rangle\leq\epsilon^\ast(v;z_j)\Big\} \enspace,
\end{align}
where we used relation~\eqref{eq::translate_func_epi} in the first equality and \eqref{eq::translate_func_conj}
in the second.

\item For $\overline{f}^\gamma(v)$:
\begin{equation}{\label{eq::thm_proof_1a}}
\begin{aligned}
 \partial_{\epsilon^\ast(v\mid\mathcal{J})}\overline{f}^\gamma(v\mid\mathcal{J}) &= \Big\{g\mid(\overline{f}_v^\gamma)^\star(g\mid\mathcal{J})\leq\epsilon^\ast(v\mid\mathcal{J})\Big\}  \\
                                            &= \Big\{g\mid(\overline{f}^\gamma)^\star(g\mid\mathcal{J})+\overline{f}^\gamma(v\mid\mathcal{J})\\
                                            &\qquad -\langle v,g\rangle\leq\epsilon^\ast(v\mid\mathcal{J})\Big\} \\
                                            &= \Big\{g\mid(\underset{j\in\mathcal{J}}{\sup}\{(\overline{f}^\gamma)^\star(g;z_j)\}+\overline{f}^\gamma(v\mid\mathcal{J}) \\
                                            &\qquad -\langle v,g\rangle\leq\epsilon^\ast(v\mid\mathcal{J})\Big\}\enspace,
\end{aligned}
\end{equation}

where we used relation~\eqref{eq::translate_func_epi} in the first equality, \eqref{eq::translate_func_conj}
in the second, and Lemma~\ref{thm::rocka_convex_hull_conjugate} in the third.
\end{itemize}
Equation~\eqref{eq::thm_proof_1a} holds for each $(\overline{f}^\gamma)^\star(g;z_j)$ individually since it holds for the pointwise supremum.
In addition, it is evident from Figure~\ref{fig::sketch2} that $\epsilon^\ast(v\mid\mathcal{J})$, expressing the distance of $\overline{f}^\gamma(v\mid\mathcal{J})$ from the Moreau envelope $f^\gamma(v)$, can be expressed with respect to any of the distances $\epsilon^\ast(v;z_j)$, reduced by a factor of $\overline{f}^\gamma(v;z_j)-\overline{f}^\gamma(v\mid\mathcal{J})$.
Equation~\eqref{eq::thm_proof_1a} thus ends up reading
\begin{equation*}
\begin{aligned}
\partial_{\epsilon^\ast(v\mid\mathcal{J})}\overline{f}^\gamma(v\mid\mathcal{J}) &= \Big\{g\mid(\overline{f}^\gamma)^\star(g;z_j)+\overline{f}^\gamma(v;z_j)\\
                                                         &\qquad-\langle v,g\rangle\leq\epsilon^\ast(v;z_j),\;\forall j\in\mathcal{J}\Big\}\enspace,
\end{aligned}
\end{equation*}
so $\partial_{\epsilon^\ast(v\mid\mathcal{J})}\overline{f}^\gamma(v\mid\mathcal{J})=\bigcap_{j\in\mathcal{J}}\partial_{\epsilon^\ast(v;z_j)}\overline{f}^\gamma(v;z_j)$ from~\eqref{eq::thm_proof_2}.

\section{Proof of Proposition~\ref{prop::inexact_DR}}{\label{app::4}}
We start by analyzing the ADMM iterations in~\eqref{eq::admm}. The proximal and dual iterations can be, respectively, expressed as
\begin{align}{\label{eq::error}}
   \hat{y}_i^{k+1} \nonumber&= v_{i}^{k} - \gamma g_{i}^{k} \\
                   \nonumber&= v_{i}^{k} - \gamma\nabla f_i^\gamma(v_{i}^{k}) - \gamma e_i^k\\
                   \nonumber&= \underbrace{\prox{_{\frac{1}{\rho} f_i}}(v_i^k)}_{y_i^{k+1}} - (1/\rho) e_i^k\\
   \hat{\lambda}_i^{k+1} \nonumber&= \lambda_i^k + \rho(\tilde{x}_i^{k+1}-\hat{y}_i^{k+1})\\
                                  &= \underbrace{\lambda_i^k + \rho(\tilde{x}_i^{k+1}-y_i^{k+1})}_{\lambda_i^{k+1}} + e_i^k              \enspace.
\end{align}
Let us introduce the variable $\mu_i^k:=\lambda_i^k+\rho \tilde{x}_i^{k+1}$ and write the dual update as
\begin{align}{\label{eq::derivation1_inexact_DR}}
\lambda_i^{k+1} \nonumber&= \mu_i^k-\rho\prox{_{\frac{1}{\rho} f_i}}(v_i^k)\\
                &= \mu_i^k-\rho\prox{_{\frac{1}{\rho} f_i}}(\mu_i^k/\rho)\enspace,
\end{align}
where the second equality follows from~\eqref{eq::prox2MoreauGrad}.

We still need to express the error in~\eqref{eq::error} with respect to the functions appearing
in~\eqref{eq::dr}.
The Moreau identity is a useful Lemma that associates a convex function with its
conjugate and is given below.
\begin{lemma}{\label{lem::Moreau}}
 Let $f:\reals^n\mapsto\reals\cup\{+\infty\}$ be a closed proper convex function. Then for any $x\in\reals^n$
 \[
  \prox{_{\rho f^{\ast}}}(x) + \rho \prox{_{f/\rho}}(x/\rho) = x, \;\; \forall\; 0 < \rho < +\infty\enspace.
 \]
\end{lemma}
Using Lemma~\ref{lem::Moreau} in~\eqref{eq::derivation1_inexact_DR}, it directly follows that
that $\lambda_i^{k+1}=\prox{_{\rho f_i^\star}}(\lambda_i^k+\rho \tilde{x}_i^{k+1})$. Substituting in~\eqref{eq::error},
we finally get that
\begin{equation}{\label{eq::lambda_approx}}
\hat{\lambda}_i^{k+1} = \prox{_{\rho f_i^\star}}(\lambda_i^k+\rho \tilde{x}_i^{k+1}) + e_i^k\enspace.
\end{equation}
It follows from the definition of $F(\lambda)$ and~\eqref{eq::lambda_approx} that Algorithm~\ref{al::algo} can be
equivalently cast as the inexact DR iteration provided that $z^{k+1}=\lambda^k+\rho \tilde{x}^{k+1}$. The proof of this last point follows directly from the slides~\cite[Lecture~13]{EE_236C}.

By switching the second and third updates, the DR iteration~\eqref{eq::dr} can be written as
\begin{align} {\label{eq::dr2}}
    r^{k+1}      \nonumber&= \prox{_{\rho H}}(2\lambda^k-z^k) \\
    \lambda^{k+1} \nonumber&= \prox{_{\rho F}}(z^k+\eta(r^{k+1}-\lambda^k)) \\
    z^{k+1}       &= z^k+\eta(r^{k+1}-\lambda^k)\enspace.
\end{align}
By introducing the variable $w=z-\lambda$, the iteration above can be cast as
\begin{align*} 
    r^{k+1}       &= \prox{_{\rho H}}(\lambda^k-w^k) \\
    \lambda^{k+1} &= \prox{_{\rho F}}(\lambda^k+w^k+\eta(r^{k+1}-\lambda^k)) \\
    w^{k+1}       &= w^k+\lambda^k+\eta(r^{k+1}-\lambda^k)\enspace.
\end{align*}

It is shown in the slides that the $x$-update of the ADMM algorithm~\eqref{eq::admm} can be expressed as $x^{k+1}=\frac{1}{\rho}(r^{k+1}+w^k-\lambda^k)$. Similarly, the $y$-update of the algorithm can be written as $y^{k+1}=\frac{1}{\rho}w^k$. Putting these two together, we have that
\begin{align*}
\tilde{x}^{k+1} &= \eta x^{k+1} + (1-\eta)y^k \\
            &= \eta\frac{1}{\rho}(r^{k+1}+w^k-\lambda^k) + (1-\eta)\frac{1}{\rho}w^k\enspace.
\end{align*}
Then $\lambda^k+\rho\tilde{x}^{k+1}=\lambda^k+w^k+\eta(r^{k+1}-\lambda^k)$, which is the $z$-update given in~\eqref{eq::dr2} after substituting $z=w+\lambda$.
\end{document}